\documentclass[11pt]{article}
\usepackage[a4paper,margin=1in]{geometry}
\usepackage[T1]{fontenc}
\usepackage[utf8]{inputenc}
\usepackage{lmodern}
\usepackage{microtype}
\usepackage{amsmath,amssymb,amsthm,mathtools}
\usepackage{enumitem}
\usepackage{mathrsfs}
\usepackage[colorlinks=true,linkcolor=blue,citecolor=blue,urlcolor=blue]{hyperref}
\newcommand{\ind}[1]{\mathbf 1_{\{#1\}}}
\allowdisplaybreaks

\theoremstyle{plain}
\newtheorem{theorem}{Theorem}[section]
\newtheorem{lemma}[theorem]{Lemma}

\newtheorem{conjecture}[theorem]{Conjecture}

\theoremstyle{definition}

\theoremstyle{remark}
\newtheorem{remark}[theorem]{Remark}

\DeclareMathOperator{\tr}{tr}

\newcommand{\eps}{\varepsilon}
\newcommand{\R}{\mathbb R}
\newcommand{\one}{\mathbf 1}

\title{
	A Matching-Number Refinement of Brouwer's Laplacian Eigenvalue Inequality\thanks{
		Email addresses: \href{mailto:jhuangmath@gzhu.edu.cn}{\texttt{jhuangmath@gzhu.edu.cn}},
		\href{mailto:chunyan_qin@126.com}{\texttt{chunyan\_qin@126.com}}.
	}
}
\author{
	Jing Huang\textsuperscript{a}
\quad
	Chunyan Qin\textsuperscript{b}
	\\[1mm]
	\textsuperscript{a}School of Mathematics and Information Science, Guangzhou University,\\
	Guangzhou 510006, China\\
	\textsuperscript{b}School of Mathematics, South China University of Technology,\\
	Guangzhou 510641, China
}
\date{}

\begin{document}
	\maketitle
	
\begin{abstract}
	Let $G=(V,E)$ be a finite simple graph with Laplacian eigenvalues
	$\lambda_1(L(G))\ge\cdots\ge\lambda_{|V|}(L(G))$, and define
	\[
	\eps_k(G)=
	\sum_{j=1}^{\min\{k,|V|\}}\lambda_j(L(G))-|E|.
	\]
	Let $\nu(G)$ be the matching number of $G$, and let $n(G)$ be the number of
	non-isolated vertices of $G$. Lew proved that
	\(\eps_k(G)\le k\nu(G)+\lfloor k/2\rfloor\), and conjectured that the
	additive term can be removed in the non-endpoint range. We prove this
	conjecture:
	\[
	\eps_k(G)\le k\nu(G)
	\qquad
	(1\le k\le n(G)-2).
	\]
	We also characterize  all equality cases. Up to isolated vertices, equality
	holds precisely for stars, for \(K_1\vee(K_k\cup\overline{K_{n-k-1}})\) with
	\(k\) odd, and for \(K_n-E(K_{1,t})\) with \(n\) odd, \(k=n-2\), and
	\(1\le t\le n-2\). We also analyze the endpoint range \(k\ge n(G)-1\), where
	\(\eps_k(G)=|E|\), and determine the specific cases where the inequality 
\(\varepsilon_k(G)\le k\nu(G)\) fails or holds with equality.
\end{abstract}

	

\section{Introduction}\label{sec:introduction}

Let \(G=(V,E)\) be a finite simple graph with \(e(G)=|E|\). Its Laplacian matrix is
\(L(G)=D(G)-A(G)\), where \(D(G)\) is the diagonal degree matrix and \(A(G)\) is
the adjacency matrix. We write the Laplacian eigenvalues in non-increasing order
as \(\lambda_1(L(G))\ge\lambda_2(L(G))\ge\cdots\ge\lambda_{|V|}(L(G))\). The
Laplacian spectrum, counted with multiplicities, is denoted by
\(\operatorname{Spec}_L(G)\), and \(\lambda^{[m]}\) means that the eigenvalue
\(\lambda\) has multiplicity \(m\). For \(k\ge0\), define
\[
\eps_k(G)=
\sum_{j=1}^{\min\{k,|V|\}}\lambda_j(L(G))-e(G),
\]
with the convention that the empty sum is \(0\). Since the sum of all Laplacian
eigenvalues is \(2e(G)\), the quantity \(\eps_k(G)\) records the excess of the
sum of the \(k\) largest Laplacian eigenvalues over the edge count.

Let \(V_+(G)=\{v\in V:\deg_G(v)>0\}\), let \(G^+=G[V_+(G)]\), and set
\(n(G)=|V_+(G)|\). Adding or deleting isolated vertices does not change
\(e(G)\), the nonzero Laplacian eigenvalues, or the matching number. Thus
\(n(G)\), rather than \(|V|\), is the relevant order parameter here. We write
\(n=n(G)\) and \(\nu=\nu(G)\), where \(\nu(G)\) is the matching number of \(G\).

A basic problem in spectral graph theory is to relate partial sums of Laplacian
eigenvalues to combinatorial parameters of the underlying graph. Brouwer's
Laplacian eigenvalue conjecture \cite{BrouwerHaemers} asserts that every finite
simple graph satisfies
\[
\eps_k(G)\le \binom{k+1}{2}
\qquad
(1\le k\le |V(G)|).
\]
The
conjecture was verified in several cases, including trees, random graphs, and
various ranges of \(k\)
\cite{BerndsenThesis,CooperConstraints,FritscherTree,HMT10,RochaRandom}, and was
recently proved in full by Kothari and Tudose \cite{KothariTudoseBrouwer}.
Related work includes refinements involving degree sums, clique number, critical
indices, product graphs, and equality cases for Brouwer-type inequalities
\cite{ChenZiFull,GaniePirzadaRatherTrevisan,GaniePirzadaTrevisanClique,
	LewApprox,LewDegreeSums,LiGuoFull,TorresTrevisanProduct,TorresTrevisanCritical}.

Brouwer's bound is universal. To obtain estimates that depend on the edge
structure of \(G\), one may ask for bounds in terms of covering or matching
parameters. Das, Mojallal, and Gutman \cite{DasMojallalGutman} proved that
\(\eps_k(G)\le k\tau(G)\), where \(\tau(G)\) is the vertex-cover number. Since
the endpoints of a maximal matching form a vertex cover, \(\tau(G)\le2\nu(G)\).
This gives a matching-number bound, but with an additional factor of two.
It is therefore natural to ask whether \(\nu(G)\) itself gives the correct
linear coefficient. Lew \cite{LewPartition}, using partition density and star
arboricity, proved
\[
\eps_k(G)
\le
k\nu(G)+\left\lfloor\frac{k}{2}\right\rfloor
\qquad
(1\le k\le |V(G)|).
\]
He conjectured that the additive term can be removed in the non-endpoint range.

\begin{conjecture}[\textup{\cite[Conjecture~5.1]{LewPartition}}]
	\label{conj:lew}
Let $G$ be a finite simple graph with $n$ non-isolated vertices.  If $1\le k\le n-2$, then
	$\eps_k(G)\le k\nu.$
\end{conjecture}

The condition \(k\le n-2\) is part of the formulation of
Conjecture~\ref{conj:lew}. If \(G\) has at least one edge, then \(G^+\) has at
most \(n-1\) nonzero Laplacian eigenvalues, and hence
$\eps_k(G)=e(G)$ when \(k\ge n-1\).
Thus, in the endpoint range, the inequality \(\eps_k(G)\le k\nu(G)\) becomes an
extremal comparison between \(e(G)\) and \(k\nu(G)\). This comparison already
fails in general at \(k=n-1\), as noted by Lew \cite{LewPartition}. Therefore
the range \(1\le k\le n-2\) is the natural range for the conjectured spectral
inequality, while the endpoint range requires a separate treatment. One purpose
of Theorem~\ref{thm:endpoint-extension} below is to determine exactly when these
endpoint failures occur.

Our first main result proves Conjecture~\ref{conj:lew} and determines the
equality cases. For integers \(k,s\ge1\), set
\(F_{k,s}=K_1\vee(K_k\cup\overline{K_s})\). For \(n\ge3\) and
\(1\le t\le n-2\), set \(D_{n,t}=K_n-E(K_{1,t})\), the graph obtained from
\(K_n\) by deleting \(t\) edges incident with a common vertex.

\begin{theorem}
	\label{thm:main}
	Let $G$ be a finite simple graph with $n$ non-isolated vertices.
	If $1\le k\le n-2$, then $
	\eps_k(G)\le k\nu.$
	Moreover, equality holds if and only if $G^+$ is one of the following
	graphs:
	\begin{enumerate}[label=\textup{(\roman*)}]
		\item the star $K_{1,n-1}$;
		\item $F_{k,n-k-1}$, where $k$ is odd;
		\item $D_{n,t}$, where $n$ is odd, $k=n-2$, and \(1\le t\le n-2\).
	\end{enumerate}
\end{theorem}

We briefly indicate the proof strategy. Edmonds' odd set cover theorem is used
to express the matching number through a covering structure consisting of
singleton vertices and pairwise disjoint odd sets. The singleton part is
controlled by the vertex-cover estimate, while the odd sets are handled using
Lew's \(k^2\) bound together with the Laplacian complement identity. These
estimates reduce the proof to one terminal case, namely a single odd packet of
order \(k+1\). This case is treated by a local absorption argument: a sparse
packet is bounded directly, while a dense packet has enough spectral gap to
absorb an external star forced by the minimality of the odd set cover. Tracking
the equality conditions in these estimates gives the three families listed in
Theorem~\ref{thm:main}.

We next turn to the endpoint range. Although \(\eps_k(G)=e(G)\) for
\(k\ge n-1\), the inequality \(\eps_k(G)\le k\nu(G)\) does not become automatic: it becomes an extremal edge-count problem for graphs with prescribed matching number and no isolated vertices. As the theorem below shows, possible failures are confined to the first endpoint \(k=n-1\). The theorem also identifies the equality cases throughout the endpoint range.

\begin{theorem}
	\label{thm:endpoint-extension}
	Let $G$ be a finite simple graph with $n$ non-isolated vertices.
	For $k\ge n-1$, the inequality
	\begin{equation}\label{eq:endpoint-bound}
		\eps_k(G)\le k\nu
	\end{equation}
	holds except precisely when $k=n-1$ and
	$G^+\cong K_{2r+1}-F,$
	where \(r\ge1\), \(F\subseteq E(K_{2r+1})\), and \(|F|<r\).
	Moreover, if $k\ge n$, then equality in \eqref{eq:endpoint-bound} holds if and only if
	\(k=n\) and \(G^+\cong K_{2r+1}\) for some \(r\ge1\).
	For $k=n-1$, equality in \eqref{eq:endpoint-bound} holds if and only if
	$G^+$ is one of the following graphs:
	\begin{enumerate}[label=\textup{(\roman*)}]
		\item $K_{1,t}$ with $t\ge1$, or $K_{2r}$ with $r\ge2$;
		\item $K_{2r+1}-F$, where $r\ge2$, \(F\subseteq E(K_{2r+1})\), and \(|F|=r\).
	\end{enumerate}
\end{theorem}

The endpoint classification is proved separately. After the reduction
\(\eps_k(G)=e(G)\) for \(k\ge n-1\), the problem becomes an extremal question for
graphs with prescribed matching number and no isolated vertices. We use the
Erd\H{o}s--Gallai theorem as the starting point, and then keep track of the
no-isolated-vertices condition in order to determine the exceptional and equality
cases. Thus Theorem~\ref{thm:main} proves the conjectured inequality in the
range \(1\le k\le n-2\), while Theorem~\ref{thm:endpoint-extension} gives the
corresponding endpoint classification.

The paper is organized as follows. Section~\ref{sec:preliminaries} collects the
spectral and matching-theoretic tools. Sections~\ref{sec:reduction} and
\ref{sec:terminal} prove the reduction and the terminal absorption estimate.
Section~\ref{sec:proof-main} proves Theorem~\ref{thm:main}, determines the
equality cases, and proves Theorem~\ref{thm:endpoint-extension}.

\section{Preliminaries}\label{sec:preliminaries}
	
All graphs in this paper are finite and simple. We collect the auxiliary tools used in the proof. We first fix the notation for
	graph unions.
	If \(G_1,\ldots,G_t\) are spanning subgraphs on a common vertex
	set \(\Omega\) with pairwise disjoint edge sets, then
	$G=G_1\cup\cdots\cup G_t$
	denotes their edge-disjoint spanning union;
	that is,
	$V(G)=\Omega, E(G)=\bigcup_{i=1}^t E(G_i).$
	If \(G_1,\ldots,G_t\) have pairwise disjoint vertex sets, then
	$G=G_1\sqcup\cdots\sqcup G_t$
	denotes their vertex-disjoint union;
	that is,
	$
	V(G)=\bigcup_{i=1}^t V(G_i),
	E(G)=\bigcup_{i=1}^t E(G_i).$
	Thus \(\sqcup\) is used whenever the pieces are vertex-disjoint; in particular,
	their edge sets are then automatically disjoint.
	\begin{lemma}[{\upshape \cite[Corollary~1]{HMT10}}]
		\label{lem:eps-subadditivity}
		Let \(G_1,\ldots,G_t\) be spanning subgraphs on a common vertex set \(V\)
		with pairwise disjoint edge sets, and let \(G=G_1\cup\cdots\cup G_t\).
		Then, for every \(1\le k\le |V|\),
		\[
		\eps_k(G)\le \sum_{i=1}^t \eps_k(G_i).
		\]
	\end{lemma}
	
	\begin{lemma}[{\upshape \cite[Corollary~2.10]{LewPartition}}]
		\label{lem:cover}
		For every graph \(G\) and every \(1\le k\le |V|\),
		\(
		\eps_k(G)\le k\tau(G),
		\)
		where \(\tau(G)\) is the vertex-cover number of \(G\).
	\end{lemma}
	
	\begin{lemma}
		\label{lem:star-bound}
		If \(T\) is a star, possibly with isolated vertices added, then
		$\eps_k(T)\le k$ 
		for every \(k\ge1\).
	\end{lemma}
	
	\begin{proof}
		Put \(d=e(T)\).
		If \(d=0\), then \(\eps_k(T)=0\).  Suppose \(d\ge1\).
		Then the nonzero Laplacian eigenvalues of \(T\) are \(d+1\) and \(1\) with
		multiplicity \(d-1\).
		Hence, for every \(k\ge1\), the sum appearing in the
		definition of \(\eps_k(T)\) is at most \(d+k\).
		Therefore
		$\eps_k(T)\le (d+k)-d=k.$
	\end{proof}
	
	We next recall Edmonds' odd set cover formulation of the matching number \cite{Edmonds}.
	For a graph \(G=(V,E)\), an \emph{odd set cover} of \(G\) is a family
	$\mathscr C=\{v_1,\ldots,v_r;S_1,\ldots,S_t\},$
	where \(v_1,\ldots,v_r\) are distinct vertices and \(S_1,\ldots,S_t\) are
	pairwise disjoint odd subsets of \(V\), such that every edge either meets
	\(\{v_1,\ldots,v_r\}\) or is contained in some \(S_i\).
	Its \emph{weight} is
	\[
	w(\mathscr C)=r+\sum_{i=1}^t\frac{|S_i|-1}{2}.
	\]
	The vertices \(v_j\) need not be disjoint from the sets \(S_i\).
	Odd sets of
	size \(1\) have weight zero and contain no edges, and will usually be left
	implicit.
	An odd set cover that minimizes this weight is called a \emph{minimum-weight odd set cover}.
	Every odd set cover \(\mathscr C\) naturally induces an edge-disjoint decomposition \(G=G_{\mathrm{vc}}\cup G_{\mathrm{odd}}\) by spanning subgraphs, where \(G_{\mathrm{odd}}\) has edge set \(\bigcup_{i=1}^t E(G[S_i])\), and \(G_{\mathrm{vc}}\) contains all remaining edges of \(G\).
	By definition, every edge of \(G_{\mathrm{vc}}\) is incident with at least one vertex in \(\{v_1,\ldots,v_r\}\).
	\begin{lemma}[{\upshape  \cite[Exercise~3.1.17]{LovaszPlummer}}]
		\label{lem:edmonds}
		Let \(G\) be a graph, and let $\mathscr C$ be a minimum-weight odd set cover of $G$. Then $\nu(G)=w(\mathscr C)$.
	\end{lemma}
	We shall also use the following extremal form of the Erd\H{o}s--Gallai
	matching theorem.
	For integers \(N,q\ge1\), let \(qK_2\) denote the vertex-disjoint union of
	\(q\) copies of \(K_2\), and define
	\[
	\operatorname{ex}(N,qK_2)
	=
	\max\{e(H): |V(H)|=N \text{ and } \nu(H)<q\}.
	\]
	
	\begin{lemma}[{\upshape \cite[Theorem~4.1]{ErdosGallai}}]
		\label{lem:erdos-gallai-matching}
		Let \(N,m\ge1\). Then
		\[
		\operatorname{ex}(N,(m+1)K_2)
		=
		\begin{cases}
			\binom N2, & N\le 2m+1,\\[2mm]
			\displaystyle
			\max\left\{
			\binom{2m+1}{2},
			\binom m2+m(N-m)
			\right\}, & N\ge 2m+1.
		\end{cases}
		\]
		For \(N\ge 2m+1\), the two terms are attained respectively by
		$K_{2m+1}\sqcup\overline{K_{N-2m-1}}$ and $
		K_m\vee\overline{K_{N-m}}.$
	\end{lemma}
	
	For a graph \(H\), let \(o(H)\) denote the number of odd components of \(H\).
	We shall use the following form of the Tutte--Berge formula.
	
	\begin{lemma}[{\upshape  \cite[Section~3.1]{LovaszPlummer}}]
		\label{lem:tutte-berge}
		If \(H\) has \(N\) vertices, then
		\[
		\nu(H)=\frac12\min_{S\subseteq V(H)}
		\bigl(N+|S|-o(H-S)\bigr).
		\]
	\end{lemma}
	
	A component is called \emph{non-trivial} if it contains at least one edge.
	A
	\emph{largest non-trivial component} is a non-trivial component with the
	maximum number of vertices.
	\begin{lemma}[{\upshape \cite[Corollary~2.3]{LewPartition}}]
		\label{lem:AM}
		Let \(G\) be a graph, and let \(n'\) be the maximum number of vertices in
		a connected component of \(G\).
		Then
		$\lambda_1(L(G))\le n'.$
	\end{lemma}
	
	In addition to the structural matching tools, our proofs rely heavily on the variational characterization of eigenvalue sums.
		For a finite set \(X\), a rank-\(k\) orthogonal projection on \(\mathbb R^X\)
	is a matrix \(P\in\mathbb R^{X\times X}\) such that
	\(P^2=P=P^{\mathsf T}\) and \(\operatorname{rank}P=k\).
	\begin{lemma}
		\label{lem:kyfan}
		Let \(A\in\mathbb R^{X\times X}\) be symmetric, where \(X\) is finite, and
		let $
		\lambda_1(A)\ge\cdots\ge\lambda_{|X|}(A)$ 
		be its eigenvalues.
		Then, for every \(1\le k\le |X|\),
		\[
		\sum_{j=1}^k\lambda_j(A)
		=
		\max_P \operatorname{tr}(PA),
		\]
		where \(P\) ranges over all rank-\(k\) orthogonal projections on
		\(\mathbb R^X\).
		Moreover, \(P\) attains the maximum if and only if
		\[
		\bigoplus_{\lambda>\lambda_k(A)}\ker(A-\lambda I)
		\subseteq \operatorname{im}P
		\subseteq
		\bigoplus_{\lambda\ge\lambda_k(A)}\ker(A-\lambda I).
		\]
	\end{lemma}

\begin{proof}
	The variational formula is Ky Fan's maximum principle
	\cite[Theorem~1]{Fan}; see also \cite{OvertonWomersley} for related
	max-characterizations of sums of largest eigenvalues. It remains to describe
	the equality case.
	Let \(z_1,\ldots,z_{|X|}\) be an orthonormal eigenbasis of \(A\), with
	\(Az_i=\lambda_i(A)z_i\). If \(P\) is a rank-\(k\) orthogonal projection, then
	\[
	\operatorname{tr}(PA)
	=
	\sum_{i=1}^{|X|}\lambda_i(A)(Pz_i,z_i)
	=
	\sum_{i=1}^{|X|}\lambda_i(A)\|Pz_i\|^2.
	\]
	Here \(0\le \|Pz_i\|^2\le1\) for all \(i\), and
	$\sum_{i=1}^{|X|}\|Pz_i\|^2=\operatorname{tr}P=k.$
	Thus \(\operatorname{tr}(PA)\) is maximized by putting the whole image of \(P\)
	inside the eigenspaces corresponding to the \(k\) largest eigenvalues. Equality
	therefore forces \(\operatorname{im}P\) to contain every eigenspace with
	eigenvalue strictly larger than \(\lambda_k(A)\), and to have no component in
	any eigenspace with eigenvalue strictly smaller than \(\lambda_k(A)\). Equivalently,
	\[
	\bigoplus_{\lambda>\lambda_k(A)}\ker(A-\lambda I)
	\subseteq \operatorname{im}P
	\subseteq
	\bigoplus_{\lambda\ge\lambda_k(A)}\ker(A-\lambda I).
	\]
	Conversely, any projection satisfying this inclusion clearly attains
	\(\sum_{j=1}^k\lambda_j(A)\). This proves the equality condition.
\end{proof}
	
	\begin{lemma}
		\label{lem:k2}
		For every graph \(G\) and every \(k\ge1\), $
		\eps_k(G)\le k^2.$
	\end{lemma}
	
	\begin{proof}
		If \(k\le |V|\), this is the graph case of \cite{LewDegreeSums}.
		If
		\(k>|V|\), then \(\eps_k(G)=e(G)\), and therefore
		\[
		\eps_k(G)=e(G)\le \binom{|V|}{2}<k^2.
		\]
	\end{proof}

	\begin{lemma}
		\label{lem:complement}
		Let \(G=(V,E)\) be a graph, and let \(\overline G\) be its complement.
		Then, for every \(0\le k\le |V|-1\),
		\[
		\eps_k(G)=\eps_{|V|-k-1}(\overline G)+|V|k-\binom{|V|}{2}.
		\]
	\end{lemma}
	
	\begin{proof}
		Write the Laplacian eigenvalues of \(G\) as
		$\lambda_1(L(G))\ge\cdots\ge\lambda_{|V|-1}(L(G))\ge\lambda_{|V|}(L(G))=0.$
		Since
		\[
		L(G)+L(\overline G)=L(K_V)=|V|I-J,
		\qquad
		J=\mathbf 1_V\mathbf 1_V^{\mathsf T},
		\]
		with \(\mathbf 1_V\) denoting the all-ones vector in \(\mathbb R^V\), we may
		choose an orthonormal eigenbasis \(\xi_1,\ldots,\xi_{|V|}\) of \(L(G)\) such
		that \(\xi_{|V|}=|V|^{-1/2}\mathbf 1_V\) and
		\(\xi_i\perp\mathbf 1_V\) for \(i<|V|\).
		If
		\(L(G)\xi_i=\lambda_i(L(G))\xi_i\) with \(i<|V|\), then \(J\xi_i=0\), and hence
		\[
		L(\overline G)\xi_i
		=
		\bigl(|V|I-J-L(G)\bigr)\xi_i
		=
		\bigl(|V|-\lambda_i(L(G))\bigr)\xi_i.
		\]
		Thus, apart from the zero eigenvalue corresponding to \(\mathbf 1_V\), the
		Laplacian eigenvalues of \(\overline G\) are
		\[
		|V|-\lambda_{|V|-1}(L(G))\ge |V|-\lambda_{|V|-2}(L(G))\ge\cdots\ge
		|V|-\lambda_1(L(G)).
		\]
		Put \(\ell=|V|-k-1\).  Then
		\[
		\sum_{j=1}^{\ell}\lambda_j(\overline G)
		=
		|V|\ell-\sum_{h=k+1}^{|V|-1}\lambda_h(L(G))
		=
		|V|(|V|-k-1)-2e(G)+\sum_{j=1}^k\lambda_j(L(G)).
		\]
		Using \(e(\overline G)=\binom{|V|}{2}-e(G)\), we get
		\[
		\eps_{|V|-k-1}(\overline G)
		=
		\eps_k(G)+|V|(|V|-k-1)-\binom{|V|}{2},
		\]
		which is equivalent to the claimed identity.
	\end{proof}

	\begin{lemma}
		\label{lem:no-isolated-matching-extremal}
		Let \(G\) be an \(N\)-vertex graph with no isolated vertices and matching number
		\(\nu\ge1\). Write \(N=2\nu+s\), where \(s\ge2\). Then
		\[
		e(G)\le
		\max\left\{
		\binom{2\nu}{2}+s,\,
		\binom{\nu}{2}+\nu(N-\nu)
		\right\}.
		\]
		If the first term is larger, equality holds precisely when 
		$G\cong K_1\vee\bigl(K_{2\nu-1}\sqcup\overline{K_s}\bigr);$
		if the second term is larger, equality holds precisely when
		$G\cong K_\nu\vee\overline{K_{N-\nu}}.$
		If the two terms are equal, equality holds precisely for these two graphs,
		which coincide when \(\nu=1\).
	\end{lemma}

	\begin{proof}
		By Lemma~\ref{lem:tutte-berge}, choose \(S\subseteq V\) such that
		\(o(G-S)-|S|=s\). Then \(G-S\) has
		\(|S|+s\) odd components. Suppose first that \(S\neq\emptyset\).
		Since \(N=2\nu+s\), after one vertex is assigned
		to each odd component, the number of remaining vertices is
		$N-|S|-(|S|+s)=2(\nu-|S|).$
		Thus \(1\le |S|\le\nu\).
		For fixed \(S\), the edge number is maximized by making \(S\) complete,
		joining it completely to \(V\setminus S\), and making the components of
		\(G-S\) complete.
		The \(|S|+s\) odd components must remain distinct, while
		any even component may be absorbed into an odd component.
		By convexity, the
		internal edge count of \(G-S\) is maximized when all remaining vertices above
		are placed in one odd component. Therefore
		\[
		e(G)\le
		\binom{|S|}{2}+|S|(N-|S|)
		+\binom{2(\nu-|S|)+1}{2}.
		\]
		The right-hand side is a strictly convex quadratic in \(|S|\), so its maximum
		on \(1\le |S|\le\nu\) is attained at an endpoint.
		Evaluating at
		\(|S|=1\) and \(|S|=\nu\), we get
		\[
		e(G)\le
		\max\left\{
		\binom{2\nu}{2}+s,\,
		\binom{\nu}{2}+\nu(N-\nu)
		\right\}.
		\]
		
		It remains to consider \(S=\emptyset\). Then \(G\) has \(s\) odd components.
		Since \(G\) has no isolated vertices, each odd component has order at least
		\(3\), this yields  \(\nu\ge s\).
		Again by convexity,
		\[
		e(G)\le \binom{2\nu-2s+3}{2}+3(s-1).
		\]
		But
		\[
		\binom{2\nu}{2}+s
		-\left(\binom{2\nu-2s+3}{2}+3(s-1)\right)
		=(2s-3)(2\nu-s)>0,
		\]
		and hence \(e(G)<\binom{2\nu}{2}+s\) in this case.
		
		Thus equality can occur only when \(S\neq\emptyset\).
		Strict convexity then
		forces \(|S|=1\) or \(|S|=\nu\). Equality in the preceding bounds also
		forces \(S\) to be complete and joined completely to \(V\setminus S\),
		and forces \(G-S\) to have exactly one nontrivial odd component, with all
		remaining odd components singletons.
		The two endpoint cases give
		$
		K_1\vee\bigl(K_{2\nu-1}\sqcup\overline{K_s}\bigr)$ and $
		K_\nu\vee\overline{K_{N-\nu}},$
		respectively. Conversely, these graphs have no isolated vertices, matching
		number \(\nu\), and exactly the corresponding two edge counts.
	\end{proof}
	
	\section{Odd Component Estimates and the Reduction Theorem}\label{sec:reduction}
	
	We next control the contribution of odd sets in a minimum-weight odd set cover.
	The
	key input is the following bound for a graph on an odd number of vertices,
	obtained from Lew's \(k^2\) theorem and the complement identity.
	We write \(\ind{\mathcal P}\) for the indicator of a statement \(\mathcal P\);
	that is, \(\ind{\mathcal P}=1\) if \(\mathcal P\) holds and
	\(\ind{\mathcal P}=0\) otherwise.
	\begin{lemma}\label{lem:one-packet}
		Let \(H\) be a graph on \(2q+1\) vertices, where \(q\ge1\).  Then, for every
		\(0\le k\le 2q+1\),
		\[
		\eps_k(H)\le kq+q\,\ind{k=2q}.
		\]
	\end{lemma}
	
	\begin{proof}
		The case \(k=0\) follows from \(\eps_0(H)=-e(H)\le0\).  If
		\(1\le k\le q\), then Lemma~\ref{lem:k2} gives
		$\eps_k(H)\le k^2\le kq.$
		
		Now assume \(q<k\le2q-1\), and put \(\ell=2q-k\).
		Then
		\(1\le \ell\le q-1\).  Since \(|V(H)|=2q+1\), Lemmas~\ref{lem:k2}
		and~\ref{lem:complement} imply
		\[
		\eps_k(H)
		=
		\eps_{\ell}(\overline H)+(2q+1)k-\binom{2q+1}{2}
		\le
		\ell^2+(2q+1)k-(2q+1)q.
		\]
		Using \(k=2q-\ell\), we obtain
		\begin{equation}\label{eq:one-packet-gap}
			kq-\bigl[\ell^2+(2q+1)k-(2q+1)q\bigr]
			=
			(\ell-1)(q-\ell).
		\end{equation}
		The right-hand side of \eqref{eq:one-packet-gap} is nonnegative for \(1\le\ell\le q-1\), and therefore
		\(\eps_k(H)\le kq\) for \(q<k\le2q-1\).
		It remains to consider \(k=2q\) and \(k=2q+1\).  Since
		\[
		\eps_{2q}(H)=\eps_{2q+1}(H)=e(H)
		\le \binom{2q+1}{2}=(2q+1)q.
		\]
		This is \(kq+q\) when \(k=2q\), and \(kq\) when \(k=2q+1\).
		The proof is
		complete.
	\end{proof}
	
	To extend the bound from a single component to the entire subgraph $G_{\mathrm{odd}}$ consisting of multiple odd sets, we need to evaluate the eigenvalue sum for vertex-disjoint unions.
	\begin{lemma}
		\label{lem:disjoint-formula}
		Let \(G_1,\ldots,G_c\) be graphs on pairwise disjoint vertex sets, and let
		\(G=G_1\sqcup\cdots\sqcup G_c\).
		Then, for every \(0\le k\le |V|\),
		\[
		\eps_k(G)
		=
		\max\left\{
		\sum_{i=1}^c \eps_{k_i}(G_i):
		0\le k_i\le |V(G_i)|,\ \sum_{i=1}^c k_i=k
		\right\}.
		\]
	\end{lemma}
	
	\begin{proof}
		The Laplacian spectrum of \(G\) is the multiset union of the Laplacian
		spectra of \(G_1,\ldots,G_c\).
		Hence the sum of the \(k\) largest Laplacian
		eigenvalues of \(G\) is obtained by choosing, for each \(i\), \(k_i\) of the
		largest eigenvalues of \(G_i\), with \(\sum_i k_i=k\), so as to maximize the
		total sum.
		Since
		$e(G)=\sum_{i=1}^c e(G_i),$
		subtracting \(e(G)\) gives the formula, with the convention
		\(\eps_0(G_i)=-e(G_i)\).
	\end{proof}
	
	\begin{lemma}\label{lem:many-packets}
		Let \(G=G_1\sqcup\cdots\sqcup G_c\) be a vertex-disjoint union, where
		\(|V(G_i)|=2q_i+1\ge3\ (i=1,2,\ldots,c)\).
		Put \(q_*=\sum_{i=1}^c q_i\).
		If \(c=1\), then $\eps_k(G)\le kq_1$
		for every \(0\le k\le |V|\) with \(k\ne2q_1\).
		If \(c\ge2\), then
		$\eps_k(G)<kq_*$
		for every \(1\le k\le |V|\). The same conclusions hold after adding
		isolated vertices to \(G\).
	\end{lemma}
	
	\begin{proof}
		Assume first that no isolated vertices are added. By
		Lemma~\ref{lem:disjoint-formula}, there exist  integers \(k_i\), with
		\(0\le k_i\le 2q_i+1\) and \(\sum_i k_i=k\), such that
		$\eps_k(G)=\sum_{i=1}^c \eps_{k_i}(G_i).$
		By Lemma~\ref{lem:one-packet},
		\[
		\eps_k(G)\le
		\sum_{i=1}^c k_iq_i+\sum_{i\in I}q_i,
		\qquad
		I=\{1\le i\le c:k_i=2q_i\}.
		\]
		
		If \(c=1\), then the second sum is nonzero only when \(k=2q_1\), which is
		excluded. Hence \(\eps_k(G)\le kq_1\).
		If \(c\ge2\), then \(q_*-q_i\ge1\) for every \(1\le i\le c\). A direct calculation gives
		\[
		\begin{aligned}
			kq_*-
			\left(\sum_{i=1}^c k_iq_i+\sum_{i\in I}q_i\right)
			&=
			\sum_{i\notin I}k_i(q_*-q_i)
			+\sum_{i\in I}q_i\bigl(2(q_*-q_i)-1\bigr).
		\end{aligned}
		\]
		The right-hand side is positive: if some \(k_i>0\) with \(i\notin I\),
		then the first sum is positive;
		otherwise, since \(k\ge1\), some
		\(i\in I\) exists, and the second sum is positive.
		Thus
		$\eps_k(G)<kq_* .$
		
		Now let \(G'\) be obtained from \(G\) by adding isolated vertices.
		By
		Lemma~\ref{lem:disjoint-formula}, there are integers
		\(k_0,k_1,\ldots,k_c\) such that
		\[
		k_0+\sum_{i=1}^c k_i=k,
		\qquad
		0\le k_i\le 2q_i+1,
		\]
		and
		\[
		\eps_k(G')
		=
		\sum_{i=1}^c \eps_{k_i}(G_i).
		\]
		The preceding argument covers the \(c=1\) case.
		For \(c\ge2\), if \(k_0>0\), then substituting \(k=k_0+\sum_{i=1}^c k_i\) gives
		\[
		\begin{aligned}
			kq_*-
			\left(\sum_{i=1}^c k_iq_i+\sum_{i\in I}q_i\right)
			&=
			k_0q_*+
			\sum_{i\notin I}k_i(q_*-q_i)
			+\sum_{i\in I}q_i\bigl(2(q_*-q_i)-1\bigr)>0.
		\end{aligned}
		\]
		If \(k_0=0\), the previous strict argument applies directly. Hence
		\(\eps_k(G')<kq_*\) for every \(k\ge1\).
	\end{proof}

	\begin{theorem}\label{thm:reduction}
		Let \(G=(V,E)\) be a graph, let \(k\ge 1\), and let \(\mathscr C=\{v_1,\ldots,v_r;S_1,\ldots,S_t\}\) be a minimum-weight odd set cover of \(G\).
		After discarding odd sets of size \(1\), relabel the remaining sets so that \(|S_i|=2q_i+1\ge 3\) for \(1\le i\le c\).
		Unless \(c=1\) and \(|S_1|=k+1\), we have $\eps_k(G)\le k\nu.$
	\end{theorem}
	
	\begin{proof}
		Let \(G=G_{\mathrm{vc}}\cup G_{\mathrm{odd}}\) be the edge-disjoint decomposition induced by \(\mathscr C\).
		Since every edge of \(G_{\mathrm{vc}}\) meets \(\{v_1,\ldots,v_r\}\), we have \(\tau(G_{\mathrm{vc}})\le r\), and Lemma~\ref{lem:cover} yields
		\begin{equation}\label{eq:eps-vc-bound}
			\eps_k(G_{\mathrm{vc}})\le kr.
		\end{equation}
		
		The subgraph \(G_{\mathrm{odd}}\) is the vertex-disjoint union of \(G[S_1],\ldots,G[S_c]\) and isolated vertices. By assumption, we cannot have \(c=1\) and \(k=2q_1\).
		Thus, Lemma~\ref{lem:many-packets} ensures that
		\begin{equation}\label{eq:eps-odd-bound}
			\eps_k(G_{\mathrm{odd}}) \le k\sum_{i=1}^c q_i,
		\end{equation}
		where the bound trivially holds as \(0\) when \(c=0\).
		Using Lemma~\ref{lem:eps-subadditivity} and substituting \eqref{eq:eps-vc-bound} and \eqref{eq:eps-odd-bound}, we obtain
		\[
		\eps_k(G) \le \eps_k(G_{\mathrm{vc}}) + \eps_k(G_{\mathrm{odd}}) \le k\left(r + \sum_{i=1}^c q_i\right).
		\]
		Finally, since \(\mathscr C\) is a minimum-weight odd set cover, Lemma~\ref{lem:edmonds} implies \(\nu = w(\mathscr C) = r + \sum_{i=1}^c q_i\), which completes the proof.
	\end{proof}

	\section{Terminal One-Packet Absorption}\label{sec:terminal}
	
	We now isolate the dense subcase of the one-packet exception left by the
	reduction.
	In the proof of the main theorem, the remaining single-packet case
	is split according to the number of edges inside the packet.
	If \(H=G[S]\)
	satisfies \(e(H)\le k^2/2\), then \(H\) is supported on \(k+1\) vertices, so
	\(\eps_k(H)=e(H)\le k^2/2\), and the edges outside \(S\) are handled by the
	vertex-cover bound.
	Thus the only subcase requiring a separate local
	absorption argument is
	$e(G[S])>\frac{k^2}{2}.$
	The purpose of this section is to prove that absorption estimate.
	All local estimates below are made on a fixed finite ambient vertex set
	\(\Omega\).
	All Laplacian and projection matrices are indexed by \(\Omega\);
	graphs defined on smaller vertex sets are regarded as graphs on \(\Omega\) by
	adding isolated vertices.
	For \(x\in\R^\Omega\), write $
	\operatorname{supp}(x)=\{w\in\Omega:x_w\ne0\}.$
	For \(X\subseteq\Omega\), define
	\[
	\mathcal R_X
	=
	\left\{
	x\in\R^\Omega:
	\operatorname{supp}(x)\subseteq X,\
	\sum_{w\in X}x_w=0
	\right\}.
	\]
	
	When \(S\subseteq\Omega\) is fixed, set
	\[
	\mathcal U=\mathcal R_S,
	\qquad
	\mathcal W=\mathcal U^\perp .
	\]
	Then \(\dim\mathcal U=|S|-1\). We denote by
	\(\Pi_{\mathcal U}\) and \(\Pi_{\mathcal W}\) the orthogonal projections onto
	\(\mathcal U\) and \(\mathcal W\), respectively.
	Thus
	$\Pi_{\mathcal U}+\Pi_{\mathcal W}=I$ and 
	$x=\Pi_{\mathcal U}x+\Pi_{\mathcal W}x$ 
	for all $x\in\R^\Omega.$
	
	\begin{lemma}\label{lem:dense-gap}
		Let \(H\) be a graph on \(S\).
		If \(e(H)=2q^2+\Delta\) with
		\(\Delta>0\), then \(1\le\Delta\le q\), and
		$x^{\mathsf T}L(H)x\ge (q+\Delta)\|\Pi_{\mathcal U}x\|^2$ for all $x\in\mathbb R^\Omega.$
	\end{lemma}
	
	\begin{proof}
		Since \(e(H)\le \binom{2q+1}{2}=2q^2+q\), we have \(1\le\Delta\le q\).
		Let
		\(\overline H\) be the complement of \(H\) on \(S\). Then \(e(\overline H)=q-\Delta\) and every component of \(\overline H\) has at
		most \(q-\Delta+1\) vertices.
		By Lemma~\ref{lem:AM},
		$\lambda_1(L(\overline H))\le q-\Delta+1.$
		For \(y\in\mathcal U\), the complement identity on \(S\) gives
		$L(H)+L(\overline H)=(2q+1)I-J,$
		and \(Jy=0\). Hence
		\[
		y^{\mathsf T}L(H)y
		=(2q+1)\|y\|^2-y^{\mathsf T}L(\overline H)y
		\ge (q+\Delta)\|y\|^2 .
		\]
		Now let \(x\in\mathbb R^\Omega\), and write \(x=y+z\), where
		\(y=\Pi_{\mathcal U}x\) and \(z=\Pi_{\mathcal W}x\). Since \(z\in\mathcal W\),
		its restriction to \(S\) is constant;
		as all edges of \(H\) lie inside \(S\),
		\(L(H)z=0\). Therefore
		\[
		x^{\mathsf T}L(H)x
		=y^{\mathsf T}L(H)y
		\ge(q+\Delta)\|y\|^2
		=(q+\Delta)\|\Pi_{\mathcal U}x\|^2 .
		\]
	\end{proof}

	The next estimate is the local absorption mechanism. It refines the usual star
	bound by keeping track of the \(\mathcal W\)-mass of the optimizing projection.
	The operator $L(T)-(2q+1)\Pi_{\mathcal W}$ penalizes precisely the part of the
	projection that is invisible to the dense packet on $S$. Thus the estimate below
	can be viewed as a tilted version of the elementary star bound: the usual error
	$e(T)$ is paid only after the projection has been charged by its
	$\mathcal W$-component. In the proof, the star is reduced to invariant blocks
	according to whether its center and leaves lie in $S$; each block has at most
	the amount of positive spectrum allowed by the number of star edges it contains.
	
	\begin{lemma}\label{lem:tilted-star}
		Let \(T\) be a star in \(\Omega\) with no edge contained in \(S\).
		Then, for
		every rank-\(2q\) orthogonal projection matrix \(P\) on \(\R^\Omega\),
		$\tr(P L(T))-e(T)
		\le (2q+1)\tr(P\Pi_{\mathcal W}).$
	\end{lemma}
	
	\begin{proof}
		Set
		\[
		m=|S|=2q+1,\qquad d=e(T),\qquad
		M_T=L(T)-m\Pi_{\mathcal W}.
		\]
		By Lemma~\ref{lem:kyfan}, it suffices to prove
		\begin{equation}\label{eq:tilted-star-target}
			\sum_{j=1}^{m-1}\lambda_j(M_T)\le d .
		\end{equation}
		Vertices outside \(S\cup V(T)\)
		contribute only eigenvalues \(-m\), so they will be ignored.
		The case
		\(d=0\) is immediate from \(M_T=-m\Pi_{\mathcal W}\).
		
		Put
		$\mathbf s=m^{-1/2}\one_S.$
		We repeatedly use the following elementary fact: if a symmetric matrix
		preserves each summand in an orthogonal decomposition, then its eigenvalues
		are the union of the eigenvalues of the corresponding blocks.
		First suppose that the center \(v\) of \(T\) lies in \(S\). Let \(Y\) be the
		leaf set.
		Since \(T\) has no edge contained in \(S\), we have
		\(Y\cap S=\varnothing\) and \(|Y|=d\).
		Let
		\[
		\mathbf p=
		\frac{\Pi_{\mathcal U}\mathbf e_v}
		{\|\Pi_{\mathcal U}\mathbf e_v\|},
		\qquad
		\mathbf y=d^{-1/2}\one_Y .
		\]
	where \(\mathbf e_v\) denotes the standard basis vector indexed by \(v\).
		Then
		\[
		\R^{S\cup Y}
		=
		\{z\in\mathcal U:z_v=0\}
		\oplus
		\mathcal R_Y
		\oplus
		\operatorname{span}\{\mathbf p,\mathbf s,\mathbf y\}
		\]
		is an orthogonal \(M_T\)-invariant decomposition.
		On the first two summands,
		\(M_T=0\) and \(M_T=(1-m)I\), respectively. On the last summand, in the
		ordered orthonormal basis \((\mathbf p,\mathbf s,\mathbf y)\), the matrix of
		\(M_T\) is
		\[
		\begin{pmatrix}
			\frac{d(m-1)}m&
			\frac{d\sqrt{m-1}}m&
			-\sqrt{\frac{d(m-1)}m}
			\\[2mm]
			\frac{d\sqrt{m-1}}m&
			\frac dm-m&
			-\sqrt{\frac dm}
			\\[2mm]
			-\sqrt{\frac{d(m-1)}m}&
			-\sqrt{\frac dm}&
			1-m
		\end{pmatrix}.
		\]
		Its characteristic polynomial is $(t-d)(t+m)(t+m-1).$
		Thus the largest \(m-1\) eigenvalues have sum \(d\), proving
		\eqref{eq:tilted-star-target} in this case.
		Now suppose that the center \(v\notin S\). Let \(Y\) be the leaf set and put
		\[
		\mathbf c=\mathbf e_v,\qquad
		Y_S=Y\cap S,\qquad
		Y_0=Y\setminus S,\qquad
		a=|Y_S|,\qquad
		b=|Y_0|.
		\]
		Then \(d=a+b\).  We split the rest of the proof according to how many leaves of
		\(T\) lie in \(S\).
		\medskip
		\noindent\textbf{Case 1: \(a=0\).} In this case, no vertex of \(T\) lies in \(S\).
		Hence \(M_T=0\) on \(\mathcal U\), while
		on the star coordinates \(M_T=L(T)-mI\). The spectrum of \(M_T\) on these coordinates is
		\[
		\left\{\,b+1-m,\;(1-m)^{[b-1]},\;-m\,\right \}.
		\]
		Together with the \(m-1\) zero eigenvalues on \(\mathcal U\), this gives
		\[
		\sum_{j=1}^{m-1}\lambda_j(M_T)
		\le \max\{b+1-m,0\}\le b=d .
		\]
		
		\medskip
		\noindent\textbf{Case 2: \(a=m\).}  In this case, \(Y_S=S\), and \(M_T=I\) on \(\mathcal U\).
		If \(b=0\), then \(d=m\), and the remaining block is
		\(\operatorname{span}\{\mathbf s,\mathbf c\}\).
		In the ordered basis
		\((\mathbf s,\mathbf c)\), this block is
		\[
		\begin{pmatrix}
			1-m&-\sqrt m\\
			-\sqrt m&0
		\end{pmatrix},
		\]
		whose characteristic polynomial is \((t-1)(t+m)\).
		Hence the largest \(m-1\)
		eigenvalues have sum \(m-1\le m=d\).
		
		If \(b>0\), set
		$\mathbf y_0=b^{-1/2}\one_{Y_0}.$
		Then \(M_T=(1-m)I\) on \(\mathcal R_{Y_0}\).
		The remaining block is
		\(\operatorname{span}\{\mathbf s,\mathbf y_0,\mathbf c\}\), where, in the
		ordered basis \((\mathbf s,\mathbf y_0,\mathbf c)\), the matrix is
		\[
		\begin{pmatrix}
			1-m&0&-\sqrt m\\
			0&1-m&-\sqrt b\\
			-\sqrt m&-\sqrt b&b
		\end{pmatrix}.
		\]
		Its characteristic polynomial is
		$(t+m)(t+m-1)(t-b-1).$
		Therefore
		\[
		\sum_{j=1}^{m-1}\lambda_j(M_T)
		\le (b+1)+(m-2)< m+b=d .
		\]
		
		\medskip
		\noindent\textbf{Case 3: \(0<a<m\).}
		Put \(S_0=S\setminus Y_S\).
		The spaces
		$\mathcal R_{Y_S},
		\mathcal R_{S_0},$
		and when \(b>0\), also \(\mathcal R_{Y_0}\), are mutually orthogonal and
		\(M_T\)-invariant.
		On these spaces,
		\[
		M_T|_{\mathcal R_{Y_S}}=I,\qquad
		M_T|_{\mathcal R_{S_0}}=0,\qquad
		M_T|_{\mathcal R_{Y_0}}=(1-m)I \quad (b>0).
		\]
		The remaining direction in \(\mathcal U\) is
		\[
		\mathbf p_0
		=
		\sqrt{\frac{m-a}{am}}\,\one_{Y_S}
		-
		\sqrt{\frac{a}{(m-a)m}}\,\one_{S_0}.
		\]
		
		First assume \(b=0\).
		The remaining block is
		\(\operatorname{span}\{\mathbf p_0,\mathbf s,\mathbf c\}\). In the ordered
		basis \((\mathbf p_0,\mathbf s,\mathbf c)\), the matrix is
		\[
		\begin{pmatrix}
			\frac{m-a}{m}&
			\frac{\sqrt{a(m-a)}}m&
			-\sqrt{\frac{a(m-a)}m}
			\\[1mm]
			\frac{\sqrt{a(m-a)}}m&
			\frac am-m&
			-\frac a{\sqrt m}
			\\[1mm]
			-\sqrt{\frac{a(m-a)}m}&
			-\frac a{\sqrt m}&
			a-m
		\end{pmatrix}.
		\]
		Its characteristic polynomial is $
		(t+m)(t-1)(t+m-a).$
		Hence the positive eigenvalues have total sum at most
		$(a-1)+1=a=d.$
		Since the sum of the largest \(m-1\) eigenvalues is bounded above by the sum of
		all positive eigenvalues, this proves \eqref{eq:tilted-star-target}.
		It remains to consider \(b>0\). Set $
		\mathbf y_0=b^{-1/2}\one_{Y_0}.$
		The remaining block is
		$\operatorname{span}\{\mathbf p_0,\mathbf s,\mathbf y_0,\mathbf c\}.$
		In the ordered basis \((\mathbf p_0,\mathbf s,\mathbf y_0,\mathbf c)\), its
		matrix is
		\[
		\begin{pmatrix}
			\frac{m-a}{m}&
			\frac{\sqrt{a(m-a)}}m&
			0&
			-\sqrt{\frac{a(m-a)}m}
			\\[1mm]
			\frac{\sqrt{a(m-a)}}m&
			\frac am-m&
			0&
			-\frac a{\sqrt m}
			\\[2mm]
			0&
			0&
			1-m&
			-\sqrt b
			\\[1mm]
			-\sqrt{\frac{a(m-a)}m}&
			-\frac a{\sqrt m}&
			-\sqrt b&
			a+b-m
		\end{pmatrix}.
		\]
		Its characteristic polynomial is \((t+m)g(t)\), where
		\begin{equation}\label{eq:gabm}
			\begin{aligned}
				g(t)
				&=t^3+(2m-2-a-b)t^2  \\
				&\quad +(m^2-3m+1-am+2a-bm+b)t
				+(m-a)(b-m+1).
			\end{aligned}
		\end{equation}
		
		Suppose that all three roots
		\(\theta_1,\theta_2,\theta_3\) of \(g(t)\) are positive.
		By Vieta's
		formula,
		\[
		\theta_1+\theta_2+\theta_3=a+b-2m+2>0,
		\qquad
		\theta_1\theta_2\theta_3=-(m-a)(b-m+1)>0.
		\]
		The first identity gives \(b>2m-2-a\ge m-1\), while the second, since
		\(0<a<m\), gives \(b<m-1\), a contradiction.
		Hence \(g(t)\) has at most
		two positive roots.
		Moreover,
		\[
		g(1)=-ab<0,
		\qquad
		g(b+1)=b(m-a)(b+m+1)>0.
		\]
		Thus \(g(t)\) has a positive root in \((1,b+1)\).
		Since its leading
		coefficient is positive and it has at most two positive roots, its largest
		positive root is at most \(b+1\).
		The positive eigenvalues of this block are precisely the positive roots of
		\(g(t)\), since the remaining root of the block polynomial is \(-m\).
		We
		show that their sum is at most \(b+1\).
		
		If \(g(t)\) has only one positive root, then this root is at most \(b+1\)
		by the preceding paragraph.
		Now suppose that \(g(t)\) has two positive
		roots, and write its roots as
		$\theta_1\ge \theta_2>0\ge \theta_3.$
		Since
		\[
		z^{\mathsf T}M_Tz
		=
		z^{\mathsf T}L(T)z-m\|\Pi_{\mathcal W}z\|^2
		\ge -m\|z\|^2,
		\]
		every eigenvalue of \(M_T\) is at least \(-m\), and hence \(\theta_3\ge -m\).
		By Vieta's formula,
		$\theta_1+\theta_2+\theta_3=a+b-2m+2.$
		Thus
		$
		\theta_1+\theta_2
		\le a+b-m+2
		\le b+1,
		$
		because \(a\le m-1\). Hence, in both cases, the last block contributes at most
		\(b+1\) to the sum of positive eigenvalues.
		Adding the \(a-1\) eigenvalues equal to \(1\) on \(\mathcal R_{Y_S}\), we get
		\[
		\sum_{\lambda_j(M_T)>0}\lambda_j(M_T)
		\le (a-1)+(b+1)=a+b=d.
		\]
		Therefore
		\[
		\sum_{j=1}^{m-1}\lambda_j(M_T)
		\le
		\sum_{\lambda_j(M_T)>0}\lambda_j(M_T)
		\le d,
		\]
		which proves \eqref{eq:tilted-star-target}.
	\end{proof}
	
	We now combine the spectral gap from Lemma \ref{lem:dense-gap} and the projection estimate from Lemma \ref{lem:tilted-star} to show that a single star can be efficiently absorbed into a dense packet.
	
	\begin{lemma}\label{lem:one-star}
		Let \(q\ge1\), let \(S\subseteq\Omega\) have size \(2q+1\), and let \(H\) be a
		graph on \(S\), viewed on \(\Omega\) by adding isolated vertices.
		If
		\(e(H)>2q^2\) and \(T\) is a non-empty star on \(\Omega\) with no edge contained
		in \(S\), then
		$\eps_{2q}(H\cup T)<2q^2+2q.$
	\end{lemma}
	
	\begin{proof}
		Write \(e(H)=2q^2+\Delta\). Then \(1\le\Delta\le q\).
		Let
		\(\mathcal U,\mathcal W\) be the subspaces associated with \(S\), so that
		\(\dim\mathcal U=2q\).
		Let \(P\in\R^{\Omega\times\Omega}\) be a rank-\(2q\) orthogonal projection, and
		set
		\[
		\alpha=\tr(P\Pi_{\mathcal W}),
		\qquad
		\beta=\tr((I-P)L(H)).
		\]
		Since \(P\) has rank \(2q\), and \(\Pi_{\mathcal U}\) is the orthogonal
		projection onto the \(2q\)-dimensional space \(\mathcal U\), we have
		$\tr P=\tr\Pi_{\mathcal U}=2q.$
		Together with \(\Pi_{\mathcal U}+\Pi_{\mathcal W}=I\), this gives
		\begin{equation}\label{eq:one-star-projection-trace}
			\alpha
			=\tr(P\Pi_{\mathcal W})=\tr P-\tr(P\Pi_{\mathcal U})
			=\tr((I-P)\Pi_{\mathcal U}).
		\end{equation}
		Let \(r_1,\ldots,r_{|\Omega|-2q}\) be an orthonormal basis of
		\(\operatorname{im}(I-P)\). Since \(I-P=\sum_i r_i r_i^{\mathsf T}\) and
		\(\Pi_{\mathcal U}\) is an orthogonal projection,
		\begin{equation}\label{eq:one-star-basis-trace}
			\sum_i\|\Pi_{\mathcal U}r_i\|^2
			=
			\sum_i r_i^{\mathsf T}\Pi_{\mathcal U}r_i
			=
			\tr((I-P)\Pi_{\mathcal U}).
		\end{equation}
		Combining \eqref{eq:one-star-projection-trace}--\eqref{eq:one-star-basis-trace} with  Lemma~\ref{lem:dense-gap}, we get
		\[
		\begin{aligned}
			\beta=\sum_i r_i^{\mathsf T}L(H)r_i 
			\ge (q+\Delta)\sum_i\|\Pi_{\mathcal U}r_i\|^2 
			=(q+\Delta)\alpha .
		\end{aligned}
		\]
		Also, since \(\tr L(H)=2e(H)\), we have
		\begin{equation}\label{eq:one-star-H-part}
			\tr(P L(H))-e(H)
			=
			e(H)-\beta
			=
			2q^2+\Delta-\beta .
		\end{equation}
		By Lemmas~\ref{lem:star-bound} and~\ref{lem:tilted-star},
		\begin{equation}\label{eq:one-star-T-part}
			\tr(P L(T))-e(T)
			\le
			\min\{2q,(2q+1)\alpha\}.
		\end{equation}
		
		We record the only numerical estimate needed below:
		\begin{equation}\label{eq:one-star-surplus}
			\Delta-\beta+\min\{2q,(2q+1)\alpha\}<2q .
		\end{equation}
		Indeed, if \(\beta>\Delta\), this is immediate.
		If \(\beta\le\Delta\), then
		\(\alpha\le \beta/(q+\Delta)\), and hence
		\[
		\begin{aligned}
			\Delta-\beta+\min\{2q,(2q+1)\alpha\}\le
			\Delta-\beta+\frac{(2q+1)\beta}{q+\Delta}
			<2q,
		\end{aligned}
		\]
		where we use \(1\le\Delta\le q\).
		Combining \eqref{eq:one-star-H-part}-\eqref{eq:one-star-surplus}, for every rank-\(2q\) orthogonal projection \(P\)
		we get
		\[
		\begin{aligned}
			\tr(P L(H\cup T))-e(H\cup T)
			=
			\tr(P L(H))-e(H)+\tr(P L(T))-e(T) <
			2q^2+2q .
		\end{aligned}
		\]
		Taking the maximum over all rank-\(2q\) orthogonal projections \(P\) and using
		Lemma~\ref{lem:kyfan}, we obtain $
		\eps_{2q}(H\cup T)<2q^2+2q.$
		
	\end{proof}

	We now apply the one-star estimate to the dense terminal packet. The minimality
	of the odd set cover ensures that every singleton cover vertex contributes at
	least one edge outside the packet, so one of these stars can be absorbed
	strictly.
	
	\begin{theorem}\label{thm:terminal}
		Let \(q\ge 1\), and let \(G=(V,E)\) have at least \(2q+2\) non-isolated
		vertices.
		Suppose that
		$\mathscr C=\{v_1,\ldots,v_r;S\}$
		is a minimum-weight odd set cover, with singleton odd sets left implicit,
		where \(r\ge 1\), \(|S|=2q+1\), and \(S\) is the only odd set of size greater
		than \(1\).
		If \(e(G[S])>2q^2\), then
		$\eps_{2q}(G)< 2q\nu.$
	\end{theorem}
	\begin{proof}
		Put \(H=G[S]\). Since \(\mathscr C\) has minimum weight, no singleton \(v_i\)
		is redundant.
		Thus, for each \(i\), there is an edge in \(E\setminus E(H)\)
		covered by \(v_i\) and by no other singleton \(v_j\).
		Choose one such edge for
		each \(v_i\). Assign every remaining edge of \(E\setminus E(H)\) to exactly
		one of its covering endpoints among \(v_1,\ldots,v_r\), and let \(T_i\) be the
		star formed by the edges assigned to \(v_i\).
		Then \(T_1,\ldots,T_r\) are
		edge-disjoint nonempty stars, \(T_i\) is centered at \(v_i\), and no edge of any
		\(T_i\) is contained in \(S\).
		By Lemma~\ref{lem:one-star}, $
		\eps_{2q}(H\cup T_1)<2q^2+2q,$
		and, for \(2\le i\le r\), Lemma~\ref{lem:star-bound} gives
		$\eps_{2q}(T_i)\le 2q.$
		Since $
		G=(H\cup T_1)\cup T_2\cup\cdots\cup T_r$
		is an edge-disjoint union, Lemma~\ref{lem:eps-subadditivity} yields
		\begin{equation}\label{eq:terminal-eps-bound}
			\begin{aligned}[b]
				\eps_{2q}(G)
				&\le \eps_{2q}(H\cup T_1)+\sum_{i=2}^r\eps_{2q}(T_i) \\
				&< (2q^2+2q)+(r-1)2q = 2q(q+r).
			\end{aligned}
		\end{equation}
		On the other hand,
		\begin{equation}\label{eq:terminal-cover-weight}
			\nu(G)=w(\mathscr C)
			=
			r+\frac{|S|-1}{2}
			=
			r+q.
		\end{equation}
		Combining \eqref{eq:terminal-eps-bound} and \eqref{eq:terminal-cover-weight},
		we obtain $\eps_{2q}(G)<2q\nu(G).$
		
	\end{proof}

	\section{Proof of the main results}\label{sec:proof-main}
	
We now prove Theorem~\ref{thm:main} and track equality. The inequality follows from the reduction theorem except for the terminal one-packet configuration, which was handled in Section~\ref{sec:terminal}. The remaining issue is to identify when equality can persist through the auxiliary estimates. We first record the required equality information for Brouwer's bound and for the star estimates.	
	\begin{lemma}[{\upshape \cite{KothariTudoseBrouwer}}]\label{lem:ceq-brouwer}
		For every finite simple graph \(H\) and every integer \(k\ge1\),
		$\eps_k(H)\le \binom{k+1}{2}.$
	\end{lemma}
	
	\begin{lemma}\label{lem:ceq-eps-one}
		For every graph \(H\),
		$\eps_1(H)\le 1.$
		Equality holds if and only if \(H\) is a nonempty star, possibly with
		isolated vertices added.
	\end{lemma}
	
	\begin{proof}
		If \(H\) has no edges, then \(\eps_1(H)=0\). Assume that \(H\) has an edge.
		Let \(H_0\) be a largest nontrivial component of \(H\), and put
		\(h=|V(H_0)|\).
		By Lemma~\ref{lem:AM}, applied componentwise,
		$\lambda_1(L(H))\le h.$
		Since \(H_0\) is connected, \(e(H_0)\ge h-1\). Hence
		\begin{equation}\label{eq:eps-one-bound}
			\eps_1(H)
			=
			\lambda_1(L(H))-e(H)
			\le
			h-e(H)
			\le
			h-e(H_0)
			\le
			1.
		\end{equation}
		
		Suppose equality holds. Then all inequalities in
		\eqref{eq:eps-one-bound} are equalities.
		Thus
		$e(H)=e(H_0)=h-1,$
		so \(H_0\) is the only nontrivial component of \(H\), and \(H_0\) is a tree.
		Also \(\lambda_1(L(H_0))=h\).
		By the eigenvalue correspondence in the proof of
		Lemma~\ref{lem:complement}, this gives a zero eigenvalue of
		\(L(\overline{H_0})\) on \(\one^\perp\). Hence
		\(\overline{H_0}\) is disconnected.
		Let \(A,B\) be a nonempty partition of \(V(H_0)\) with no edges of
		\(\overline{H_0}\) between \(A\) and \(B\).
		Then all edges between \(A\) and
		\(B\) belong to \(H_0\). Since \(H_0\) is a tree, the complete bipartite
		graph between \(A\) and \(B\) is acyclic.
		Hence \(\min\{|A|,|B|\}=1\).
		These cross edges already contribute \(h-1\) edges, and \(e(H_0)=h-1\).
		Therefore \(H_0\) is a star.
		Conversely, if \(H\) is \(K_{1,h-1}\), possibly with isolated vertices
		added, then $\eps_1(H)=1.$
	\end{proof}
Building on these extremal properties for $k=1$, we can now sharply characterize the equality cases for a single odd packet.
	
	\begin{lemma}
		\label{lem:ceq-odd-packet-equality}
		Let \(H\) be a graph on a fixed set \(S\) of size \(2q+1\), where
		\(q\ge 1\).
		Let \(k\ge 1\) and \(k\ne 2q\). Then
		$\eps_k(H)\le kq.$
		Moreover, equality holds only in the following cases:
		\begin{enumerate}[label=\textup{(\roman*)}]
			\item \(k=2q+1\), and \(H=K_S\), where \(K_S\) denotes  \(K_{|S|}\) on the vertex set \(S\).
			\item \(k=2q-1\), and \(\overline H\) is a nonempty star, possibly with
			isolated vertices added.
		\end{enumerate}
	\end{lemma}
	
	\begin{proof}
		First let \(1\le k\le 2q-2\).
		By Lemma~\ref{lem:ceq-brouwer},
		$\eps_k(H)\le \binom{k+1}{2}<kq.$
		If \(k>2q+1\), then all nonzero Laplacian eigenvalues of \(H\) are included.
		Thus
		\[
		\eps_k(H)=e(H)\le \binom{2q+1}{2}<kq.
		\]
		
		If \(k=2q+1\), then again \(\eps_k(H)=e(H)\), and hence
		$\eps_k(H)\le \binom{2q+1}{2}=kq.$
		Equality holds if and only if \(H=K_S\).
		
		It remains to consider \(k=2q-1\).
		By Lemma~\ref{lem:complement} and  Lemma~\ref{lem:ceq-eps-one},
		\begin{equation*}
			\begin{aligned}[b]
				\eps_{2q-1}(H)=
				\eps_1(\overline H)+(2q+1)(2q-1)-\binom{2q+1}{2}\le (2q-1)q .
			\end{aligned}
		\end{equation*}
		Equality holds if and only if \(\eps_1(\overline H)=1\), equivalently,
		\(\overline H\) is a nonempty star, possibly with isolated vertices added.
	\end{proof}

	\begin{lemma}
		\label{lem:ceq-common-maximizer}
		Let \(G_1,\ldots,G_s\) be spanning graphs on a common vertex set \(\Omega\),
		with pairwise disjoint edge sets, and put \(G=G_1\cup\cdots\cup G_s\).
		If
		$\eps_k(G)=\sum_{i=1}^s\eps_k(G_i),$
		then there exists a rank-\(k\) orthogonal projection \(P\) on \(\R^\Omega\)
		such that $\tr(P L(G_i))-e(G_i)=\eps_k(G_i)$ for every $i.$
	\end{lemma}
	
	\begin{proof}
		By Lemma~\ref{lem:kyfan}, let \(P\) be a rank-\(k\) orthogonal projection
		attaining \(\eps_k(G)\).
		Since the edge sets are pairwise disjoint,
		\[
		L(G)=\sum_{i=1}^s L(G_i),
		\qquad
		e(G)=\sum_{i=1}^s e(G_i).
		\]
		Hence
		\begin{equation}\label{eq:common-maximizer-chain}
			\begin{aligned}[b]
				\eps_k(G)=
				\sum_{i=1}^s\bigl(\tr(P L(G_i))-e(G_i)\bigr)
				\le
				\sum_{i=1}^s\eps_k(G_i).
			\end{aligned}
		\end{equation}
		Since 	$\eps_k(G)=\sum_{i=1}^s\eps_k(G_i)$ and 
		$\tr(P L(G_i))-e(G_i)\le \eps_k(G_i)$ for every $i,$
		each term in \eqref{eq:common-maximizer-chain} must attain equality.
		Therefore
		$\tr(P L(G_i))-e(G_i)=\eps_k(G_i)$ for every $i.$
	\end{proof}
	To apply Lemma \ref{lem:ceq-common-maximizer} to stars, we introduce specific vectors associated with their eigenspaces. 
	For a non-empty star \(T\) with center \(v\) and leaf set \(Y\), set
	$z_T=|Y|\mathbf e_v-\sum_{y\in Y}\mathbf e_y.$
	
	\begin{lemma}
		\label{lem:ceq-star-maximizer}
		Let \(T\) be a nonempty star on \(\Omega\), with center \(v\) and leaf set
		\(Y\), and let \(k\ge 1\).
		If \(\eps_k(T)=k\), then \(|Y|\ge k\).
		Moreover, if a rank-\(k\) orthogonal projection \(P\) satisfies
		$\tr(P L(T))-e(T)=k,$
		then $z_T\in\operatorname{im}P\subseteq\mathcal R_{\{v\}\cup Y}.$
	\end{lemma}
	
	\begin{proof}
		The nonzero Laplacian eigenvalues of \(T\) are
		\[
		|Y|+1,\qquad 1^{[|Y|-1]},
		\]
		with eigenspaces \(\operatorname{span}(z_T)\) and \(\mathcal R_Y\).
		Hence
		$\eps_k(T)=\min\{k,|Y|\}.$
		Thus \(\eps_k(T)=k\) gives \(|Y|\ge k\).
		
		Let \(P\) be a rank-\(k\) orthogonal projection satisfying
		$\tr(P L(T))-e(T)=k.$
		Then \(\tr(P L(T))=\sum_{j=1}^k\lambda_j(L(T))\).
		Since \(|Y|\ge k\), the \(k\)-th largest
		Laplacian eigenvalue of \(T\) is \(1\). By Lemma~\ref{lem:kyfan},
		\[
		\operatorname{span}(z_T)
		\subseteq
		\operatorname{im}P
		\subseteq
		\operatorname{span}(z_T)\oplus\mathcal R_Y
		=
		\mathcal R_{\{v\}\cup Y}.
		\]
		Therefore $z_T\in\operatorname{im}P\subseteq\mathcal R_{\{v\}\cup Y}.$
	\end{proof}

	\begin{lemma}
		\label{lem:ceq-two-stars-strict}
		Let \(T_1,\ldots,T_r\) be edge-disjoint stars on a common vertex set, with
		distinct prescribed centers, and put \(T=T_1\cup\cdots\cup T_r\).
		Then
		$\eps_k(T)\le rk$
		for every \(k\ge 1\). Equality holds if and only if \(r=1\) and \(T_1\) is a
		nonempty star with at least \(k\) leaves.
	\end{lemma}
	
	\begin{proof}
		By Lemma~\ref{lem:eps-subadditivity} and Lemma~\ref{lem:star-bound},
		\[
		\eps_k(T)\le \sum_{i=1}^r \eps_k(T_i)\le rk.
		\]
		Suppose equality holds.
		Then
		$\eps_k(T)=\sum_{i=1}^r \eps_k(T_i)$ and $
		\eps_k(T_i)=k$ for every $i.$
		In particular, each \(T_i\) is nonempty.
		If \(r\ge 2\), Lemma~\ref{lem:ceq-common-maximizer}
		gives a rank-\(k\) orthogonal projection \(P\) maximizing \(T_i\) and \(T_j\)
		for any distinct \(i,j\).
		Write \(v_i,v_j\) for the centers and \(Y_i,Y_j\) for the leaf sets. By
		Lemma~\ref{lem:ceq-star-maximizer},
		\[
		z_{T_i}\in\operatorname{im}P\subseteq \mathcal R_{\{v_j\}\cup Y_j},
		\qquad
		z_{T_j}\in\operatorname{im}P\subseteq \mathcal R_{\{v_i\}\cup Y_i}.
		\]
		The vector \(z_{T_i}\) is nonzero at every vertex of \(T_i\).
		Since
		\(z_{T_i}\in \mathcal R_{\{v_j\}\cup Y_j}\), all these vertices must lie in
		\(\{v_j\}\cup Y_j=V(T_j)\). Hence \(V(T_i)\subseteq V(T_j)\). Similarly,
		\(V(T_j)\subseteq V(T_i)\). Hence \(V(T_i)=V(T_j)\).
		The centers are
		distinct, so the edge \(v_iv_j\) belongs to both stars, contradicting
		edge-disjointness. Thus \(r=1\).
		
		When \(r=1\), equality is exactly \(\eps_k(T_1)=k\).
		By
		Lemma~\ref{lem:ceq-star-maximizer}, this holds if and only if \(T_1\) has at
		least \(k\) leaves.
	\end{proof}
	
	\begin{lemma}
		\label{lem:ceq-complete-packet-star}
		Let \(S\subseteq\Omega\) have size \(k\ge 2\), and let \(T\) be a nonempty
		star with no edge contained in \(S\).
		If $
		\eps_k(K_S\cup T)=\eps_k(K_S)+k,$
		then the center of \(T\) lies outside \(S\), and every vertex of \(S\) is a
		leaf of \(T\).
	\end{lemma}
	\begin{proof}
		By Lemmas~\ref{lem:eps-subadditivity} and~\ref{lem:star-bound},
		\[
		\eps_k(K_S\cup T)
		\le
		\eps_k(K_S)+\eps_k(T)
		\le
		\eps_k(K_S)+k.
		\]
		Hence equality forces \(\eps_k(T)=k\) and $\eps_k(K_S\cup T)
		=\eps_k(K_S)+\eps_k(T)$.
		By Lemma~\ref{lem:ceq-common-maximizer}, there is a
		rank-\(k\) orthogonal projection \(P\) such that \(P\) maximizes both
		\(K_S\) and \(T\).
		Write \(v\) for the center of \(T\) and \(Y\) for its leaf set.
		For \(K_S\), the eigenspace with positive eigenvalue \(k\) is
		\[
		\mathcal U
		=
		\left\{
		x\in\R^\Omega:
		\operatorname{supp}(x)\subseteq S,\ 
		\sum_{w\in S}x_w=0
		\right\}.
		\]
		Since \(P\) maximizes \(K_S\), Lemma~\ref{lem:kyfan} gives
		$\mathcal U\subseteq\operatorname{im}P.$
		Since \(P\) maximizes \(T\), Lemma~\ref{lem:ceq-star-maximizer} gives
		$\operatorname{im}P\subseteq\mathcal R_{\{v\}\cup Y}.$
		Thus
		\begin{equation}\label{eq:complete-packet-star-containment}
			\mathcal U\subseteq\mathcal R_{\{v\}\cup Y}.
		\end{equation}
		
		We claim that \(S\subseteq \{v\}\cup Y\). If not, choose
		\(u\in S\setminus(\{v\}\cup Y)\) and \(w\in S\setminus\{u\}\).
		Then
		\(\mathbf e_u-\mathbf e_w\in\mathcal U\), but
		\(\mathbf e_u-\mathbf e_w\notin\mathcal R_{\{v\}\cup Y}\), contradicting
		\eqref{eq:complete-packet-star-containment}. Hence \(S\subseteq V(T)\).
		If \(v\in S\), then some vertex of \(S\setminus\{v\}\) is a leaf of \(T\),
		because \(k\ge 2\).
		This gives an edge of \(T\) contained in \(S\), contrary
		to the assumption. Therefore \(v\notin S\).
		Since \(S\subseteq\{v\}\cup Y\),
		every vertex of \(S\) lies in \(Y\).
	\end{proof}
	
	\begin{lemma}
		\label{lem:ceq-codim-two-strict}
		Let \(S\subseteq\Omega\) have size \(2q+1\), where \(q\ge 1\).
		Let \(H\)
		be a graph on \(S\) such that \(\overline H\) is a nonempty star, possibly
		with isolated vertices added.
		If \(T\) is a nonempty star with no edge
		contained in \(S\), then
		$\eps_{2q-1}(H\cup T)<\eps_{2q-1}(H)+2q-1.$
	\end{lemma}
	
	\begin{proof}
		Write \(v'\) for the center of the star component of \(\overline H\), write
		\(Y'\) for its leaf set, and put \(d=|Y'|\).
		Set
		$z_{\overline H}=d\mathbf e_{v'}-\sum_{y\in Y'}\mathbf e_y.$
		On
		\[
		\mathcal U
		=
		\left\{
		x\in\R^\Omega:
		\operatorname{supp}(x)\subseteq S,\ 
		\sum_{w\in S}x_w=0
		\right\},
		\]
		the  identity $L(H)+L(\overline H)=L(K_S)=(2q+1)I-J$ gives
		\begin{equation}\label{eq:codim-two-complement}
			L(H)x=(2q+1)x-L(\overline H)x .
		\end{equation}
		The vector \(z_{\overline H}\) is the top eigenvector of
		\(L(\overline H)\) on \(\mathcal U\), with eigenvalue \(d+1\).
		All other
		eigenvalues of \(L(\overline H)\) on \(\mathcal U\) are at most \(1\).
		Hence
		the unique maximizing subspace for \(\eps_{2q-1}(H)\) is
		\begin{equation}\label{eq:codim-two-max-subspace}
			\mathcal M_H
			=
			\mathcal U\cap\operatorname{span}(z_{\overline H})^\perp .
		\end{equation}
		Thus every rank-\((2q-1)\) projection maximizing \(H\) has image
		\(\mathcal M_H\).
		In particular, this image is supported on \(S\).
		By Lemmas~\ref{lem:eps-subadditivity} and~\ref{lem:star-bound},
		\[
		\eps_{2q-1}(H\cup T)
		\le
		\eps_{2q-1}(H)+\eps_{2q-1}(T)
		\le
		\eps_{2q-1}(H)+2q-1.
		\]
		It remains to rule out equality.
		Suppose that
		$\eps_{2q-1}(H\cup T)=\eps_{2q-1}(H)+2q-1.$
		Then $\eps_{2q-1}(H\cup T)
		=
		\eps_{2q-1}(H)+\eps_{2q-1}(T)$ and
		$\eps_{2q-1}(T)=2q-1.$
		Lemma~\ref{lem:ceq-common-maximizer} gives a rank-\((2q-1)\) orthogonal
		projection \(P\) maximizing both \(H\) and \(T\).
		The uniqueness of the maximizing
		subspace for \(H\) and Lemma~\ref{lem:ceq-star-maximizer} give
		\begin{equation}\label{eq:codim-two-contradiction}
			z_T\in\operatorname{im}P=\mathcal M_H.
		\end{equation}
		But \(T\) has no edge contained in \(S\), so \(z_T\) has a nonzero coordinate
		outside \(S\). Hence \(z_T\notin\mathcal M_H\), contradicting \eqref{eq:codim-two-contradiction}.
	\end{proof}
	
	\begin{proof}[\textup{\textbf{Proof of Theorem~\ref{thm:main}}}]
		Deleting isolated vertices does not change \(e(G)\), \(n\), \(\nu\), or
		\(\eps_k(G)\).
		Thus we may assume that \(G\) has no isolated vertices. Then
		\(|V|=n\) and \(k\le n-2\).
		Let $\mathscr C=\{v_1,\ldots,v_r;S_1,\ldots,S_c\}$
		be a minimum-weight odd set cover, with singleton odd sets left implicit, and
		write \(|S_i|=2q_i+1\ge 3\).
		If \(c\ne 1\), or if \(c=1\) and
		\(k\ne 2q_1\), Theorem~\ref{thm:reduction} gives the desired inequality.
		
		It remains to consider \(c=1\) and \(k=2q_1\).
		Put
		\[
		q=q_1,\qquad S=S_1,\qquad H=G[S].
		\]
		Then \(|S|=k+1=2q+1\). If \(e(H)\le kq\), then \(H\) is supported on
		\(k+1\) vertices, so \(\eps_k(H)=e(H)\le kq\).
		All edges outside \(H\) meet
		\(\{v_1,\ldots,v_r\}\), and hence form a graph with vertex cover of size at
		most \(r\).
		By Lemmas~\ref{lem:eps-subadditivity},~\ref{lem:cover}, and \ref{lem:edmonds},
		$\eps_k(G)\le kq+kr=k(q+r)=k\nu.$
		
		If \(e(H)>kq=k^2/2\), then \(r\ge 1\); otherwise all edges, and hence all
		vertices, would lie in \(S\), contradicting \(k\le n-2\).
		Thus
		Theorem~\ref{thm:terminal} gives
		$\eps_k(G)=\eps_{2q}(G)<2q\nu=k\nu.$
		This proves the inequality.
		
		We now characterize equality. Since isolated vertices are irrelevant, it is
		enough to classify the no-isolated core.
		Assume
		$\eps_k(G)=k\nu.$
		
		Let $\mathscr C=\{v_1,\ldots,v_r;S_1,\ldots,S_c\}$
		be a minimum-weight odd set cover, with singleton odd sets discarded, and
		write
		\[
		|S_i|=2q_i+1\ge 3\ (1\le i\le c),
		\qquad
		q_*=\sum_{i=1}^c q_i.
		\]
		Use the decomposition
		$G=G_{\mathrm{odd}}\cup G_{\mathrm{vc}},$
		and decompose \(G_{\mathrm{vc}}\) into edge-disjoint stars centered at
		\(v_1,\ldots,v_r\). By Lemma~\ref{lem:edmonds},
		$\nu=q_*+r.$ 
		
		We first show  \(c=1\).
		Suppose that  \(c\ge 2\). Then Lemma~\ref{lem:many-packets} gives
		$\eps_k(G_{\mathrm{odd}})<kq_*,$
		while Lemma~\ref{lem:cover} gives \(\eps_k(G_{\mathrm{vc}})\le kr\). Hence
		\[
		\eps_k(G)
		\le
		\eps_k(G_{\mathrm{odd}})+\eps_k(G_{\mathrm{vc}})
		<
		k(q_*+r)
		=
		k\nu,
		\]
		a contradiction.

		If \(c=0\), then \(G=G_{\mathrm{vc}}\) and \(\nu=r\).
		By
		Lemma~\ref{lem:ceq-two-stars-strict}, equality in \(\eps_k(G)\le kr\)
		forces \(r=1\), and \(G\) is a nonempty star with at least \(k\) leaves.
		Since \(G\) has no isolated vertices, $G\cong K_{1,n-1}.$
		
		It remains to consider \(c=1\). Write
		\[
		S=S_1,\qquad |S|=2q+1,\qquad H=G[S],
		\]
		so \(\nu=q+r\).
		
		First suppose \(k=2q\).
		Since \(|S|=k+1\le n-1\) and \(G\) has no isolated
		vertices, \(G_{\mathrm{vc}}\) is nonempty. If \(e(H)<kq\), then
		\[
		\eps_k(G)
		\le
		\eps_k(H)+\eps_k(G_{\mathrm{vc}})
		<
		kq+kr
		=
		k\nu,
		\]
		a contradiction.
		If \(e(H)>kq\), Theorem~\ref{thm:terminal} gives
		\(\eps_k(G)<k\nu\), again a contradiction. If  \(e(H)=kq\), then \(H\) is connected;
		otherwise
		$e(H)\le \binom{2q}{2}<2q^2=kq.$
		Hence \(L(H)\) has exactly \(k\) positive eigenvalues, and its positive
		eigenspace is
		\[
		\mathcal U
		=
		\left\{
		x\in\R^\Omega:
		\operatorname{supp}(x)\subseteq S,\ 
		\sum_{w\in S}x_w=0
		\right\}.
		\]
		Equality in
		\[
		\eps_k(G)
		\le
		\eps_k(H)+\eps_k(G_{\mathrm{vc}})
		\le
		kq+kr
		=
		k\nu
		\]
		forces equality throughout. In particular,
		$\eps_k(G_{\mathrm{vc}})=kr.$
		Since \(G_{\mathrm{vc}}\) is the union of the prescribed stars centered at
		\(v_1,\ldots,v_r\), Lemma~\ref{lem:ceq-two-stars-strict} gives \(r=1\).
		Thus \(G_{\mathrm{vc}}\) itself is a nonempty prescribed star \(T\).
		Lemma~\ref{lem:ceq-common-maximizer} gives a rank-\(k\) projection \(P\)
		maximizing both \(H\) and \(T\).
		Lemmas~\ref{lem:kyfan} and \ref{lem:ceq-star-maximizer} give
		\begin{equation}\label{eq:critical-packet-star-contradiction}
			z_T\in\operatorname{im}P=\mathcal U.
		\end{equation}
		But \(T\) has no edge contained in \(S\), so
		\(z_T\) has a nonzero coordinate outside \(S\), contradicting \eqref{eq:critical-packet-star-contradiction}.
		Therefore equality is impossible when \(k=2q\).
		
		If \(k>2q+1\), then
		$\eps_k(H)=e(H)\le \binom{2q+1}{2}<kq,$
		so equality is impossible. Thus \(1\le k\le 2q+1\) and \(k\ne 2q\).
		Equality in
		\begin{equation}\label{eq:main-one-packet-chain}
			\eps_k(G)\le
			\eps_k(H)+\eps_k(G_{\mathrm{vc}})
			\le
			kq+kr
			=
			k\nu
		\end{equation}
		forces equality throughout. In particular, \(\eps_k(H)=kq\). By
		Lemma~\ref{lem:ceq-odd-packet-equality}, there are two cases.
		
		First let \(k=2q+1=|S|\) and \(H=K_S\).
		Equality in
		\eqref{eq:main-one-packet-chain} gives \(\eps_k(G_{\mathrm{vc}})=kr\), so
		Lemma~\ref{lem:ceq-two-stars-strict} gives \(r=1\). Thus
		\(G_{\mathrm{vc}}=T\) is a nonempty star.
		Also
		$\eps_k(K_S\cup T)=\eps_k(K_S)+k.$
		Since \(T\) has no edge contained in \(S\), Lemma~\ref{lem:ceq-complete-packet-star}
		implies that the center of \(T\) lies outside \(S\), and every vertex of
		\(S\) is a leaf.
		Since \(G\) has no isolated vertices, every vertex outside
		\(S\) other than the center is also a leaf.
		Hence
		$G\cong
		K_1\vee\bigl(K_k\cup\overline{K_{n-k-1}}\bigr)
		=
		F_{k,n-k-1}.$
		
		Second let \(k=2q-1=|S|-2\), and let \(\overline H\) be a nonempty star,
		possibly with isolated vertices added.
		If \(G_{\mathrm{vc}}\) is nonempty,
		choose a nonempty prescribed star \(T\). By Lemma~\ref{lem:ceq-codim-two-strict},
		$\eps_k(H\cup T)<\eps_k(H)+k.$
		Adding the remaining stars by Lemmas~\ref{lem:eps-subadditivity}
		and~\ref{lem:star-bound}, we get
		\[
		\eps_k(G)<\eps_k(H)+rk=k(q+r)=k\nu,
		\]
		a contradiction.
		Hence \(G_{\mathrm{vc}}\) is empty. Then \(S\) covers all
		edges, so minimality gives \(r=0\). Since \(G\) has no isolated vertices,
		\(V=S\).
		Thus $G=H, n=|S|=k+2,$
		and \(\overline G\) is a nonempty star with isolated vertices.
		The deleted
		star has \(t\) edges with \(1\le t\le n-2\), so
		$G\cong K_n-E(K_{1,t})= D_{n,t}.$
		This proves necessity.
		
		Conversely, we check the listed graphs.
		For \(G=K_{1,n-1}\),
		$\operatorname{Spec}_L(G)=\left\{n,1^{[n-2]},0\right\}.$
		Thus, for \(1\le k\le n-2\),
		$\eps_k(G)=(n+k-1)-(n-1)=k=k\nu.$
		
		For $G=F_{k,n-k-1}$  with \(k\) odd, we have
		\[
		\operatorname{Spec}_L(G)=\left\{n,(k+1)^{[k-1]},1^{[s]},0\right\},
		\qquad
		e(G)=(n-1)+\binom{k}{2}, \qquad \nu=\frac{k+1}{2}.
		\]
		Hence
		$
		\eps_k(G)
		=
		n+(k-1)(k+1)-(n-1)-\binom{k}{2}
		=
		k\nu.$
		
		Finally, let $
		G=D_{n,t},$ where $
		n=2r+1,$ and $
		1\le t\le n-2,$
		and let \(k=n-2\). Then  \(\nu=r\) and  \(\eps_1(\overline G)=1\).
		By
		Lemma~\ref{lem:complement},
		$
		\eps_{n-2}(G)
		=
		1+n(n-2)-\binom n2
		=
		(n-2)\nu.$
		Thus every listed graph attains equality.
	\end{proof}

	\begin{remark}\textup{
			The listed families are not disjoint: \(F_{1,n-2}\cong K_{1,n-1}\), and
			\(F_{k,1}\cong D_{k+2,k}\) for odd \(k\);
			in particular,
			\(K_{1,2}\cong F_{1,1}\cong D_{3,1}\).}
	\end{remark}

	\begin{proof}[\textup{\textbf{Proof of Theorem~\ref{thm:endpoint-extension}}}]
		Deleting or adding isolated vertices does not change $e(G)$, $n(G)$, $\nu(G)$, or $\eps_k(G)$.
		Hence we may assume without loss of generality that $G$ has no isolated vertices.
		Then $G=G^+$, $n=|V|$, and we have
		\begin{equation}\label{eq:endpoint-eps}
			\eps_k(G)=e(G) \qquad (k\ge n-1).
		\end{equation}
		
		We first record the consequence of Lemma~\ref{lem:erdos-gallai-matching} that governs the endpoint range.
		Since $n\ge 2\nu$, the case $n=2\nu$ trivially yields $e(G)\le\binom{2\nu}{2}<n\nu$. If $n\ge 2\nu+1$, Lemma~\ref{lem:erdos-gallai-matching} gives
		\begin{equation}\label{eq:endpoint-eg-bound}
			e(G)\le
			\max\left\{
			\binom{2\nu+1}{2},
			\binom{\nu}{2}+\nu(n-\nu)
			\right\}
			\le n\nu.
		\end{equation}
		Moreover, equality in \eqref{eq:endpoint-eg-bound} occurs if and only if $n=2\nu+1$ and $G\cong K_{2\nu+1}$.
		
		First suppose $k\ge n$.
		By \eqref{eq:endpoint-eps} and \eqref{eq:endpoint-eg-bound}, we obtain
		$\eps_k(G)=e(G)\le n\nu\le k\nu.$
		Since $\nu\ge 1$, equality holds if and only if $k=n$ and $e(G)=n\nu$.
		By the preceding observation, this is equivalent to $k=n$ and $G\cong K_{2\nu+1}$, which proves the assertion for $k\ge n$.
		Now suppose $k=n-1$. The inequality \eqref{eq:endpoint-bound} reduces to $e(G)\le (n-1)\nu$.
		If this inequality fails, the upper bound in \eqref{eq:endpoint-eg-bound} must strictly exceed $(n-1)\nu$.
		Because the second  term is always at most $(n-1)\nu$, the first term must be the strict maximum, which forces $n=2\nu+1$.
		Consequently, $e(G)>2\nu^2=\binom{2\nu+1}{2}-\nu$, meaning $G\cong K_{2\nu+1}-F$ for some $F\subseteq E(K_{2\nu+1})$ with $|F|<\nu$. 
		
		Conversely, let $G\cong K_{2r+1}-F$ with $|F|<r$.
		If $r=1$, then $G\cong K_3$ and the inequality fails. If $r\ge 2$, assume for contradiction that $\nu(G)\le r-1$.
		Applying Lemma~\ref{lem:erdos-gallai-matching} with matching number $r-1$ gives
		\[
		e(G)\le
		\max\left\{
		\binom{2r-1}{2},
		\binom{r-1}{2}+(r-1)(r+2)
		\right\}<2r^2,
		\]
		which contradicts $e(G)=\binom{2r+1}{2}-|F|>2r^2$. Thus $\nu(G)=r$, yielding $\eps_{n-1}(G)=e(G)>2r^2=(n-1)\nu(G)$. This completely characterizes the failure cases.
		It remains to classify equality at $k=n-1$, corresponding to $e(G)=(n-1)\nu$. If $\nu=1$, every two edges intersect.
		Since $G$ lacks isolated vertices, $G$ is either a star or $K_3$.
		The latter satisfies $e(K_3)>2\nu(K_3)$, so equality uniquely identifies the stars.
		
	Now	assume $\nu\ge 2$. If $n=2\nu$, then $e(G)=(2\nu-1)\nu=\binom{2\nu}{2}$, forcing $G\cong K_{2\nu}$.
 If $n=2\nu+1$, then $e(G)=2\nu^2=\binom{2\nu+1}{2}-\nu$, forcing $G\cong K_{2\nu+1}-F$ with $|F|=\nu$.
 If $n\ge 2\nu+2$, write $n=2\nu+s$ with $s\ge2$. Lemma~\ref{lem:no-isolated-matching-extremal} gives
			\[
			e(G)\le
			\max\left\{
			\binom{2\nu}{2}+s,\,
			\binom{\nu}{2}+\nu(n-\nu)
			\right\}.
			\]
			Since $\nu\ge2$,
			$\binom{2\nu}{2}+s< (n-1)\nu$
			and $\binom{\nu}{2}+\nu(n-\nu)<(n-1)\nu.$
			Thus equality is impossible for $n\ge 2\nu+2$.
		
		Conversely, stars and the graphs $K_{2r}$ clearly satisfy $e(G)=(n-1)\nu(G)$.
		Finally, let $H=K_{2r+1}-F$ with $r\ge 2$ and $|F|=r$. Then $e(H)=2r^2$. If $\nu(H)\le r-1$, Lemma~\ref{lem:erdos-gallai-matching} yields $e(H)<2r^2$, a contradiction.
		Hence $\nu(H)=r$, confirming $e(H)=2r^2=(|V(H)|-1)\nu(H)$. Restoring the deleted isolated vertices yields the stated classification.
	\end{proof}

\end{document}